\documentclass[leqno,12pt]{article}

\usepackage{amssymb,amscd}

\usepackage{amsthm}

\theoremstyle{plain} 

\newtheorem{theorem}{\sc Theorem}[section]

\newtheorem{lemma}[theorem]{\sc Lemma}

\newtheorem{corollary}[theorem]{\sc Corollary}

\newtheorem{proposition}[theorem]{\sc Proposition}

\theoremstyle{definition}

\newtheorem{definition}[theorem]{\sc Definition}

\newtheorem{remark}[theorem]{\sc Remark}

\newtheorem{example}[theorem]{\sc Example}

\newtheorem{remarks}[theorem]{\sc Remarks}

\newtheorem{examples}[theorem]{\sc Examples}

\newcommand\bC{{\mathbb C}}

\newcommand\bG{{\mathbb G}}

\newcommand\bP{{\mathbb P}}

\newcommand\bZ{{\mathbb Z}}

\newcommand\cG{{\mathcal G}}
\newcommand\cH{{\mathcal H}}

\newcommand\cK{{\mathcal K}}

\newcommand\cO{{\mathcal O}}

\newcommand\aff{{\rm aff}}

\newcommand\ant{{\rm ant}}

\newcommand\gp{{\rm gp}}
\newcommand\id{{\rm id}}

\newcommand\red{{\rm red}}

\newcommand\At{{\rm At}}
\newcommand\Aut{{\rm Aut}}
\newcommand\End{{\rm End}}
\newcommand\Ext{{\rm Ext}}
\newcommand\GL{{\rm GL}}
\newcommand\Hom{{\rm Hom}}
\newcommand\Int{{\rm Int}}
\newcommand\Ker{{\rm Ker}}
\newcommand\Lie{{\rm Lie}}

\newcommand\SL{{\rm SL}}
\newcommand\SO{{\rm SO}}
\newcommand\Spec{{\rm Spec}}

\makeatletter

\def\address#1#2{\begingroup
\noindent\parbox[t]{7.8cm}{%
\small{\scshape\ignorespaces#1}\par\vskip1ex
\noindent\small{\itshape E-mail address}%
\/: #2\par\vskip4ex}\hfill%
\endgroup}

\makeatother

\title{Homogeneous bundles over abelian varieties}

\author{Michel Brion}

\date{}

\begin{document}

\maketitle


\begin{abstract}
We obtain characterizations and structure results for homogeneous principal bundles over abelian
varieties, that generalize work of Miyanishi and Mukai on homogeneous vector bundles. For this,
we rely on notions and methods of algebraic transformation groups, especially observable subgroups 
and anti-affine groups. 
\end{abstract}


\section{Introduction}
\label{sec:introduction}

A vector bundle $E$ over an abelian variety $A$ is called homogeneous 
(or translation-invariant), if $E$ is isomorphic to it pull-back under any translation 
in $A$. This defines a remarkable class of vector bundles: for example, the homogeneous 
line bundles are exactly the algebraically trivial ones; they are parametrized by the 
dual abelian variety $\widehat{A}$. Also, vector bundles over an elliptic curve are close
to being homogeneous; specifically, each indecomposable vector bundle of degree $0$ is 
homogeneous (as follows from \cite{At57a}). 
This result does not extend to abelian varieties of higher dimensions, but the homogeneous 
vector bundles still admit a nice description in that generality. Namely, by work of 
Miyanishi and Mukai (see \cite{Mi73, Muk79}), the following conditions are equivalent
for a vector bundle $E$ over $A$:

\noindent
{\rm (i)} $E$ is homogeneous.

\noindent
{\rm (ii)} $E$ is an iterated extension of algebraically trivial line bundles.

\noindent
{\rm (iii)} $E \cong \bigoplus_i L_i \otimes U_i$, 
where each $L_i$ is an algebraically trivial line bundle, and each $U_i$ is a
unipotent vector bundle, i.e., an iterated extension of trivial line bundles.

Similar results had been obtained earlier by Matsushima and Morimoto in the setting of 
holomorphic vector bundles over complex tori. They also showed that the homogeneous bundles
are exactly those admitting a connection; then they admit an integrable connection 
(see \cite{Ma59, Mo59}, and \cite{Bi04, BG08} for more recent developments).

In the present article, we obtain somewhat analogous characterizations and structure results
for those principal bundles $\pi : X \to A$ under an algebraic group $G$, 
that are homogeneous in the same sense. This holds e.g. if $X$ is a commutative algebraic group; 
then $\pi$ is a homomorphism with kernel $G$. In fact, 
by a classical result of Rosenlicht and Serre (see \cite{Se59}), all principal bundles under 
the additive group are obtained in this way, as well as all homogeneous bundles under 
the multiplicative group; the latter bundles are parametrized by $\widehat{A}$. 

Actually, any homogeneous bundle under an arbitrary algebraic group may be obtained 
from an extension of $A$, in view of our main result:

\begin{theorem}\label{thm:main}
The following assertions are equivalent for a principal bundle $\pi: X \to A$ 
under an algebraic group $G$:

\smallskip

\noindent
{\rm (a)} $\pi$ is homogeneous.

\smallskip

\noindent
{\rm (b)} There exist a commutative subgroup scheme $H \subset G$ and an 
extension of commutative group schemes $0 \to H \to \cG \to A \to 0$
such that $\pi$ admits a reduction of structure group to the $H$-bundle 
$\cG \to A$. Moreover, there is a smallest such subgroup scheme $H$, 
up to conjugacy in $G$.

\smallskip

\noindent
{\rm (c)} The associated bundle $X/R_u(G) \to X/G_{\aff}$ is homogeneous,
where $R_u(G)$ denotes the unipotent radical, and $G_{\aff}$ the largest connected 
affine subgroup of $G$.

\smallskip

In characteristic $0$, these assertions are also equivalent to:

\smallskip

\noindent
{\rm (d)} $\pi$ admits a connection.

\smallskip

\noindent
{\rm (e)} $\pi$ admits an integrable connection.
\end{theorem}

When $G$ is the general linear group $\GL_n$, the $G$-bundles correspond to 
the vector bundles of rank $n$; it is easy to see that this correspondence preserves 
homogeneity. In this setting, Theorem \ref{thm:main} generalizes the above results of 
Miyanishi and Mukai. Indeed, (b) implies that any homogeneous vector bundle $E$ of rank 
$n$ is associated with a representation $\rho : H \to \GL_n$, where $H$ is a commutative 
group scheme. Then the image of $H$ is contained (up to conjugacy) in the subgroup of 
upper triangular matrices; thus, $E$ is an iterated extension of line bundles. 
The latter are homogeneous by (c) and hence algebraically trivial, which yields (ii). 
Likewise, (iii) follows from the decomposition of $\rho$ into a direct sum of weight spaces 
for the diagonalizable part of $H$. Conversely, if a vector bundle $E$ satisfies (ii) 
or (iii), then it is homogeneous by (c).

For an arbitrary algebraic group $G$, Theorem \ref{thm:main} reduces somehow the 
study of principal $G$-bundles to the case that $G$ is connected and reductive. 
One may obtain a further reduction to irreducible bundles in a suitable sense, 
and explicitly describe these bundles; this is developed in \cite{Br11a} 
for the projective linear group (equivalently, for projective bundles), and in 
\cite{Br11b} for the projective orthogonal and symplectic groups.

The proof of Theorem \ref{thm:main} is based on an analysis of the group scheme 
of equivariant automorphisms, $\Aut^G(X)$. 
This approach, already followed by Miyanishi in \cite{Mi73}, is combined 
here with recent developments on the structure of (possibly non-affine) algebraic groups, 
see \cite{Br09, SS09, Br10a, Br10b}. Namely, the group scheme $\Aut^G(X)$ is locally 
of finite type, and hence contains a largest subgroup scheme $\cG$ which is anti-affine 
(i.e., every global regular function on $\cG$ is constant). Then $\cG$ is a connected 
commutative algebraic group that acts on $A$ by translations via the natural homomorphism 
$\pi_*: \Aut^G(X) \to \Aut(A)$. One easily shows that $\pi$ is homogeneous if and only if 
$\pi_*$ maps $\cG$ onto $A$ (Proposition \ref{prop:aa}); equivalently, 
$G \times \cG$ acts transitively on $X$. From this, (a)$\Rightarrow$(c) 
and (a)$\Leftrightarrow$(b) follow without much difficulty, where $H$ is the kernel 
of $\pi_{*\vert \cG}$. To prove (c)$\Rightarrow$(a), we use a characterization 
of homogeneity in terms of associated vector bundles which is of independent interest
(Theorem \ref{thm:ass}) together with Mukai's result that (iii)$\Rightarrow$(i).
The characterization in terms of connections also follows readily from Proposition 
\ref{prop:aa}; this is carried out in Proposition \ref{prop:conn}. 

Actually, our approach yields a classification of homogeneous $G$-bundles over $A$
in terms of pairs $(H,\cG)$, where $H \subset G$ is a commutative subgroup scheme,
and $\cG$ an anti-affine extension of $A$ by $H$ (Theorem \ref{thm:str}). 
Such extensions have been classified in \cite{Br09, SS09} when $H$ is a connected affine
algebraic group; the case that $H$ is an (affine) group scheme is obtained in Theorem 
\ref{thm:cla} along similar lines. But classifying all commutative subgroup schemes 
of an algebraic group $G$ up to conjugacy is a difficult open problem, 
already when $G$ is the general linear group.

\bigskip

\noindent
{\bf Acknowledgements.}
This work was completed during a staying at the Institute of Mathematical Sciences,
Chennai, in January 2011. I thank that institute and the Chennai Mathematical 
Institute for their hospitality, and their members for fruitful discussions. 
In particular, I thank V.~Balaji for having informed me of the unpublished preprint
``Translation-Invariant and Semi-stable Sheaves on Abelian Varieties'' after I gave 
him an almost final version of the present paper. That preprint was written by 
V.B. Mehta and M.V. Nori around 1984: it contains a proof of the above results of 
Matsushima, Miyanishi, Morimoto and Mukai that goes along similar lines as the one 
presented here. Specifically, Mehta and Nori obtain statements equivalent 
to Proposition \ref{prop:aa}, Corollary \ref{cor:uni} and Proposition \ref{prop:con};
they also sketch a classification of anti-affine extensions, similar to Theorem 
\ref{thm:cla}.

\bigskip 

\noindent
{\bf Notation and conventions.}
We will consider algebraic varieties and schemes over a fixed algebraically closed field 
$k$ of arbitrary characteristic. Unless explicitly mentioned, schemes under consideration 
are assumed to be of finite type over $k$; by a point of 
such a scheme $X$, we mean a closed point. A \emph{variety} is a reduced and separated 
scheme; in particular, varieties need not be irreducible. 

We will use the book \cite{DG70} as a general reference for \emph{group schemes}. 
Given a group scheme $G$, we denote by $e_G \in G(k)$ the neutral element
and by $G_0$ the \emph{neutral component}, i.e, the connected component of $G$
that contains $e_G$. Then $G_0$ is a connected normal subgroup scheme of $G$, 
and the quotient $G/G_0$ is finite (under our assumption that $G$ is of finite type). 
A homomorphism of group schemes $f: G \to H$ is \emph{faithful} if its (scheme-theoretic)
kernel is trivial. 

Throughout this article, $G$ denotes an \emph{algebraic group}, that is, a smooth group 
scheme, and $A$ an \emph{abelian variety}, i.e., a connected complete algebraic group. 
Then $A$ is commutative, hence its group law will be denoted additively, and the neutral 
element by $0_A$. For any point $a \in A$, we denote by $\tau_a : A \to A$ the translation
$x \mapsto x + a$. Our general reference for abelian varieties is the book 
\cite{Mum70};  for affine or equivalently, linear algebraic groups, we refer to \cite{Sp98}.

\section{Characterizations}
\label{sec:ch}

\subsection{Reduction to the case of a connected affine structure group}
\label{subsec:red}

A variety $X$ equipped with a $G$-action $G \times X \to X$, $(g,x) \mapsto g \cdot x$ 
and with a $G$-invariant morphism
\begin{equation}\label{eqn:prin}
\pi : X \longrightarrow A
\end{equation}
is a \emph{(principal) $G$-bundle}, also known as a $G$-\emph{torsor}, 
if $\pi$ is faithfully flat and the morphism 
$G \times X \to X \times_A X$, $(g,x) \mapsto  (x, g \cdot x)$ is an isomorphism.
Then the morphism $\pi$ is smooth, and hence so is $X$. Moreover, $\pi$ is locally isotrivial
and quasi-projective (see e.g. \cite[Corollary 3.5 and Remark 3.6 (iii)]{Br10b}). Therefore, 
the variety $X$ is quasi-projective as well. If $G$ is connected, then $X$ is irreducible.

\begin{definition}\label{def:prin}
The $G$-bundle (\ref{eqn:prin}) is called \emph{homogeneous}, if the translation 
$\tau_a : A \to A$ lifts to a $G$-automorphism $\widetilde{\tau_a}: X \to X$
for any point $a \in A$.
\end{definition}

Denote by $\pi_a : \tau_a^*(X) \to A$ the pull-back of the $G$-bundle (\ref{eqn:prin}) 
under $\tau_a$, so that we have a cartesian square
$$
\CD
\tau_a^*(X) @>{\psi_a}>> X \\
@V{\pi_a}VV @V{\pi}VV \\
A  @>{\tau_a}>> A. \\
\endCD
$$
Then the $G$-equivariant lifts $\widetilde{\tau_a}$ of $\tau_a$ correspond bijectively to 
the isomorphisms of $G$-bundles $\varphi_a : X \to  \tau_a^*(X)$. Thus, 
\emph{a $G$-bundle is homogeneous if and only if it is isomorphic to its pull-back under 
any translation in $A$}.

We now reformulate the notion of homogeneity in terms of the \emph{equivariant automorphism group} 
$\Aut^G(X)$. By \cite[Theorem 4.2]{Br10b}, the group functor $S \mapsto \Aut^G_S(X \times S)$ 
is represented by a group scheme, locally of finite type, that we still denote by $\Aut^G(X)$ 
for simplicity. Moreover, we have an exact sequence
\begin{equation}\label{eqn:sch}
\CD
1 @>>> \Aut^G_A(X) @>>> \Aut^G(X) @>{\pi_*}>> \Aut(A)
\endCD
\end{equation}
where $\Aut^G_A(X)$ denotes the group scheme of \emph{bundle automorphisms},
$S \mapsto \Aut^G_{A \times S}(X\times S)$, and $\pi_*$ arises from the homomorphisms 
$\Aut^G_S(X \times S) \to \Aut_S(A \times S)$. Also, by a standard result 
(see e.g. \cite[Lemma 4.1]{Br10b}), the map 
$$
\Hom^G(X,G) \longrightarrow \Aut^G_A(X), \quad 
f \longmapsto (x \mapsto f(x) \cdot x)
$$
is an isomorphism, where $\Hom^G(X,G)$ denotes the group functor
$S \mapsto \Hom^G(X \times S, G)$ (the group of $G$-equivariant morphisms 
$X \times S \to G$ for the given $G$-action on $X$, the trivial action on $S$ and the 
conjugation action on $G$). 

If $G$ is affine, then so is $\Aut^G_A(X)$; if $G$ is the general linear group, then 
$\Aut^G_A(X)$ is a connected algebraic group. But for an arbitrary $G$, the group scheme
$\Aut^G_A(X)$ need not be connected nor reduced (see Examples \ref{ex:eagst}). 

As in \cite{Mi73}, we may view $\Aut^G(X)$ as the scheme of pairs 
$(\varphi,\psi)$, where $\varphi \in \Aut(A)$ and $\psi : X \to \varphi^*(X)$ is an isomorphism 
of $G$-bundles; then $\pi_*$ is just the projection $(\varphi,\psi) \mapsto \varphi$. 

The group scheme $\Aut(A)$ is the semi-direct product of the normal subgroup
$A$ of translations with the subgroup scheme $\Aut_{\gp}(A)$ of automorphisms of the algebraic
group $A$. Moreover, $A = \Aut_0(A)$, and $\Aut_{\gp}(A)$ is the constant group scheme associated
with a subgroup of $\GL_N(\bZ)$ for some integer $N \geq 1$. Clearly, 
\emph{(\ref{eqn:prin}) is homogeneous if and only if the image of $\pi_*$ contains $A$}. 

Next, recall that the group scheme $\Aut^G(X)$ admits a largest subgroup scheme $\cG$ 
which is \emph{anti-affine}, i.e., $\cO(\cG) = k$. Moreover, $\cG$ is a normal connected 
commutative algebraic subgroup, and the quotient $\Aut^G(X)/\cG$ is the \emph{affinization} 
of $\Aut^G(X)$ (the largest affine quotient group scheme), see \cite[III, \S 3, no. 8]{DG70}. 
We set 
$$
\cG = \Aut^G(X)_{\ant}.
$$
Note that $\cG$ is mapped to a subgroup of $A$ under $\pi_* : \Aut^G(X) \to \Aut(A)$.
 
Also, recall that $X$ admits a largest anti-affine group of automorphisms, that we denote 
by $\Aut(X)_{\ant}$; moreover, this subgroup centralizes any connected algebraic group 
of automorphisms (see \cite[Lemma 3.1]{Br10a} when $X$ is irreducible; 
the case of an arbitary variety $X$ follows readily, since any irreducible component of $X$
is stable under a connected group of automorphisms). In particular, $\cG \subset \Aut(X)_{\ant}$ 
with equality if $G$ is connected. 

We may now state our key observation:

\begin{proposition}\label{prop:aa}
With the above notation, the $G$-bundle (\ref{eqn:prin}) is homogeneous if and only 
if $\pi_*$ maps $\cG$ onto $A$. 

If in addition $G$ is affine, then $\cG$ is the smallest subgroup scheme of $\Aut^G(X)$ 
that is mapped onto $A$ via $\pi_*$. 
\end{proposition}

\begin{proof}
Assume that (\ref{eqn:prin}) is homogeneous. Then $\pi_*$ restricts to a surjective
morphism of group schemes $\pi_*^{-1}(A) \to A$. Let $H \subset \Aut^G(X)$ denote the reduced 
neutral component of $\pi_*^{-1}(A)$. Then $\pi_*$ still maps $H$ onto $A$, and $H$ contains 
$\cG$ as a normal subgroup such that $H/\cG$ is affine. It follows that $A/\pi_*(\cG)$, a quotient 
of $H/\cG$, is affine as well. But $A/\pi_*(\cG)$ is complete as a quotient of $A$. We conclude 
that $\pi_*(\cG) = A$. The converse is obvious.

Next, assume that $G$ is affine.
Let $\cH \subset \Aut^G(X)$ be a subgroup scheme such that $\pi_*(\cH) = A$ and consider the
largest anti-affine subgroup $\cH_{\ant} \subset \cH$. Then, as above, $A/\pi_*(\cH_{\ant})$ is 
a quotient of the affinization $\cH/\cH_{\ant}$ and hence $\pi_*(\cH_{\ant}) = A$. Since 
$\cH_{\ant} \subset \cG$, it follows that $\cG = \cK \cH_{\ant}$ where 
$\cK := \cG \cap \Aut^G_A(X)$. Then $\cK$ is affine; thus, $\cG/\cH_{\ant}$ is affine 
(as a quotient of $\cK$) and anti-affine (as a quotient of $\cG)$, and hence trivial.
We conclude that $\cG = \cH_{\ant} \subset \cH$.
\end{proof}

\begin{remark}\label{rem:red}
For an arbitrary $G$-bundle (\ref{eqn:prin}), the central subgroup $G_{\ant} \subset G$ 
acts on $X$ by bundle automorphisms. In fact, \emph{the induced map
$G_{\ant} \to \Aut^G_A(X)_{\ant}$ is an isomorphism}.

Indeed, for any $x \in X$, the evaluation homomorphism
$$
\Aut^G_A(X) \cong \Hom^G(X,G) \longrightarrow G, \quad u \longmapsto u(x)
$$
sends $\Hom^G(X,G)_{\ant}$ to $G_{\ant}$. Thus,
$$
\Aut^G_A(X)_{\ant} \cong \Hom^G(X,G)_{\ant} \subset \Hom^G(X,G_{\ant}) = \Hom(A,G_{\ant})
$$
where the latter equality holds since $G$ acts trivially on $G_{\ant}$. Moreover, for any
morphism $f: A \to G_{\ant}$, the map $a \mapsto f(a) f(0)^{-1}$ is a group homomorphism,
and hence 
$$
\Hom(A,G_{\ant}) \cong G_{\ant} \ltimes \Hom_{\gp}(A,G_{\ant}).
$$
Also, one easily shows that $\Hom_{\gp}(A,G_{\ant})$ is a free abelian group of finite rank. 
Thus, $\Hom(A,G_{\ant})_0 = G_{\ant}$ which yields our claim.

If the bundle (\ref{eqn:prin}) is trivial, then the exact sequence (\ref{eqn:sch}) splits.
Hence the group scheme $\Aut^G(X)$ is the semi-direct product of $\Aut(A)$ with the normal
subgroup scheme
$$
\Aut^G_A(X) \cong \Hom^G(G \times A, G) \cong \Hom(A, G).
$$
Moreover, as above, 
$$
\Hom(A,G)\cong G \ltimes \Hom_{\gp}(A, G) \cong G \ltimes \Hom_{\gp}(A, G_{\ant}).
$$ 
It follows that $\cG \cong G_{\ant} \times A$. 
\end{remark}

As a first application of Proposition \ref{prop:aa}, we reduce the study of homogeneous 
torsors to the case that $X$ is \emph{irreducible}. For an arbitrary $G$-torsor 
(\ref{eqn:prin}), note that $G$ acts transitively on the set of irreducible components of
$X$. Choose such a component $X_1$ and denote by $G_1 \subset G$ its stabilizer. 
Then $G_1$ is a closed subgroup containing $G_0$, and $X \cong G \times^{G_1} X_1$,
i.e., the map $G \times X_1 \to X$, $(g,x_1) \mapsto g \cdot x_1$ is a $G_1$-bundle 
for the action defined by
$g_1 \cdot (g,x_1) = (gg_1^{-1}, g_1 \cdot x_1)$; moreover, the restriction 
$\pi_1 := \pi_{\vert X_1} : X_1 \to A$ is a $G_1$-bundle as well.

\begin{proposition}\label{prop:conn}
With the above notation, $\pi: X \to A$ is homogeneous if and only if so is
$\pi_1: X_1 \to A$. 
\end{proposition} 

\begin{proof}
The connected group $\cG$ stabilizes every component of $X$, and hence restricts to
an anti-affine subgroup of $\Aut^{G_1}(X_1)$. In view of Proposition \ref{prop:aa}, it
follows that $\pi_1$ is homogeneous if so is $\pi$. For the converse, note that there is
a natural homomorphism $\Phi : \Aut^{G_1}(X_1) \to \Aut^G(X)$ which is compatible with 
the homomorphisms to $\Aut(A)$ (see Proposition \ref{prop:ext} for the construction of 
$\Phi$ in a more general setting).
\end{proof}

Next, we assume that $X$ is irreducible, and reduce to $G$ being \emph{connected}. 
For this, we consider the factorization
\begin{equation}\label{eqn:conn}
\CD
X @>{\pi_0}>> Y = X/G_0 @>{\varphi}>> A
\endCD
\end{equation}
of the $G$-bundle (\ref{eqn:prin}), where $G_0 \subset G$ stands for the neutral 
component, $\pi_0$ is a $G_0$-bundle, and $\varphi$
is a $G/G_0$-bundle. (The existence of (\ref{eqn:conn}) follows e.g. from 
\cite[Proposition 7.1]{MFK94}). Then $Y$ is an irreducible variety and $\varphi$
is a finite \'etale covering. By the Serre-Lang theorem (see e.g. \cite[IV.18]{Mum70}),
$Y$ is an abelian variety, once a base point $0_Y$ is chosen in the fiber of $\varphi$
at $0_A$; moreover, $\varphi$ is a separable isogeny.

\begin{theorem}\label{thm:conn}
With the above notation and assumptions, $\pi$ is homogeneous if and only if so is
$\pi_0$.
\end{theorem}

\begin{proof}
If $\pi$ is homogeneous, then $\pi_*(\cG) = A$ by Proposition \ref{prop:aa}. Moreover,
$\cG \subset \Aut^{G_0}(X)$ acts on $Y$ via $(\pi_0)_*$. As $\cG$ is a connected 
algebraic group, it follows that $\cG$ is mapped onto $Y$ by $(\pi_0)_*$, i.e., 
$\pi_0$ is homogeneous.

For the converse, consider the group
$$
\cG_o := \Aut^{G_0}(X)_{\ant}
$$
on which $G$ acts by conjugation via its finite quotient group $F := G/G_0$. By assumption,
$(\pi_0)_*$ maps $\cG_o$ onto $Y$, and hence $\pi_*$ maps $\cG_o$ onto $A$; moreover,
$\pi_*$ is invariant under  the action of $F$. Given $a \in A$, let $u \in \cG_o$ 
such that $\pi_*(u) = a$; then $\pi_*(\prod_{\gamma \in F} \gamma \cdot u) = N a$
where $N$ denotes the order of $F$. Thus, $Na$ admits a lift in the fixed point 
subgroup $\cG_o^F$. But $\cG_o^F \subset \Aut^G(X)$ and the multiplication 
$N : A \to A$ is surjective; hence $\pi$ is homogeneous.
\end{proof} 

We now assume that $G$ is connected, and reduce to $G$ being \emph{affine}. 
By Chevalley's structure theorem (see \cite{Co02} for a modern proof), 
$G$ admits a largest connected affine subgroup $G_{\aff}$; 
moreover, $G_{\aff}$ is a normal subgroup of $G$ and $G/G_{\aff}$ is an
abelian variety. Note that $G = G_{\aff} G_{\ant}$, where $G_{\ant}$ denotes the 
largest anti-affine subgroup of $G$ (indeed, $G_{\aff} G_{\ant}$ is a closed normal subgroup
of $G$, and the quotient $G/G_{\aff} G_{\ant}$ is both affine and complete). Also,
recall from \cite[Corollary 3.5]{Br10b} that the $G$-bundle (\ref{eqn:prin}) has
a unique factorization
\begin{equation}\label{eqn:aff}
\CD
X @>{p}>> Y = X/G_{\aff} @>{q}>> A = X/G
\endCD
\end{equation}   
where $p$ is a $G_{\aff}$-bundle and $q$ is a $G/G_{\aff}$-bundle. Then 
\emph{$Y$ is an abelian variety and $q$ is a homomorphism of algebraic groups}, 
again once an origin $0_Y$ is chosen in the fiber of $q$ at $0_A$. 
Indeed, we have the following result, which is well-known in characteristic $0$ 
(see \cite[Corollaire 1]{Se58}):

\begin{proposition}\label{prop:ab}
Let $\pi : X \to A$ be a principal bundle under an abelian variety $G$. Then 
$X$ is an abelian variety as well. Moreover, $G$ acts on $X$ via an exact sequence
of abelian varieties
$$
0 \longrightarrow G \longrightarrow X \longrightarrow A \longrightarrow 0.
$$
\end{proposition} 

\begin{proof}
Recall that $X \cong G \times^H Y$ for some finite subgroup scheme $H \subset G$ 
and some closed $H$-stable subscheme $Y \subset X$; moreover, 
the restriction $\pi_{\vert Y} : Y \to A$ is an $H$-bundle (see e.g. 
\cite[Corollary 3.5]{Br10b}). By \cite[Remark 2]{No83}), we 
have $Y \cong H \times^K B$ for some closed subgroup scheme $K \subset H$ 
and some abelian variety $B$ fitting into an exact sequence
$$
0 \longrightarrow K \longrightarrow B \longrightarrow A \longrightarrow 0.
$$
Thus,
$$
X \cong G \times^H (H \times^K B) \cong G \times^K B
$$
and $\pi$ is identified with the morphism $p : G \times^K B \to B/K = A$. Moreover,
$G \times^K B = (G \times B)/K$ is an abelian variety, and $p$ is a homomorphism
with kernel $G \times^K K \cong G$.  
\end{proof}

\begin{theorem}\label{thm:aff}
The following assertions are equivalent for a principal bundle $\pi: X \to A$
under a connected algebraic group $G$:

\smallskip

\noindent
{\rm (i)} $\pi$ is homogeneous.

\smallskip

\noindent
{\rm (ii)} The associated $G_{\aff}$-bundle $p : X \to Y$ is homogeneous
(where $Y$ is an abelian variety by Proposition \ref{prop:ab}).
\end{theorem}

\begin{proof}
The group $\cG \subset \Aut^{G_{\aff}}(X)$ acts on $Y$ via
$p_*: \Aut^{G_{\aff}}(X) \to \Aut(Y)$. If $\pi$ is homogeneous, then this
action is transitive. Indeed, $\cG$ acts transitively on $A$ by Proposition
\ref{prop:aa}. Also, $\cG$ contains the largest anti-affine subgroup
$G_{\ant} \subset G$; moreover, $G_{\ant}$ acts trivially on $A$ and transitively 
on the fibers of the $G/G_{\aff}$-bundle $q : Y \to A$, since the natural 
homomorphism $G_{\ant} \to G/G_{\aff}$ is surjective. As $\cG$ is connected 
and $Y$ is an abelian variety, it follows that $p_*$ induces a surjective 
homomorphism $\cG \to Y$. Thus, $p$ is homogeneous if so is $\pi$.

The converse implication is obvious, since every translation of $A$ 
lifts to a translation of~$Y$.     
\end{proof}

\subsection{Change of structure group}
\label{subsec:csg}

Given a $G$-bundle $\pi: X \to S$ over an arbitrary base, recall that a 
\emph{reduction} to a subgroup scheme $H \subset G$ is a closed $H$-stable 
subscheme $Y \subset X$ such that the natural map $G \times^H Y \to X$
is an isomorphism. Then $Y$ is the fiber of the associated $G$-equivariant 
morphism $f : X \to G/H$ at the base point of $G/H$. In fact, the data of
$Y$ and of $f$ are equivalent, and $\pi$ restricts to an $H$-bundle
$Y \to S$ (see e.g. \cite[Lemma 3.2]{Br10b}). 

Also, given a homomorphism of algebraic groups $f : G \to G'$, we may consider the $G'$-bundle 
$X' := G' \times^G X \to S$ obtained by \emph{extension of structure group}. This bundle
might not always exist in this generality (i.e., $X'$ may not be a scheme), but it does
exist in our setting, and this construction preserves homogeneity:

\begin{proposition}\label{prop:ext}
Let $\pi: X \to A$ be a $G$-bundle, and $f : G \to G'$ a homomorphism of algebraic groups.

\smallskip

\noindent
{\rm (i)} The $G'$-bundle 
$$
\pi': X' := G' \times^G X \longrightarrow A
$$
exists, and there is a natural homomorphism
$$
\Phi: \Aut^G(X) \longrightarrow \Aut^{G'}(X'),
$$
compatible with $\pi_*: \Aut^G(X) \to \Aut(A)$ and with 
$\pi'_*: \Aut^{G'}(X') \to \Aut(A)$. In particular, $\pi'$ is homogeneous
if so is $\pi$.

\smallskip

\noindent
{\rm (ii)} Conversely, if $\pi'$ is homogeneous, then so is $\pi$ under the
additional assumptions that $f$ is faithful and $G'/f(G)$ is quasi-affine.
\end{proposition}  

\begin{proof}
(i) Consider the neutral component $G_0 \subset G$ and the (finite) quotient
group $F := G/G_0$. Then the $G'$-bundle $G' \times^{G_0} X \to X/G_0$ exists
by \cite[Corollary 3.4]{Br10b}; moreover, $X/G_0$ is a finite cover of $A$, 
and hence is projective. Thus, the variety $G'\times^{G_0} X$ 
is quasi-projective (see \cite[Corollary 3.5]{Br10b}; it carries an action of $F$ 
arising from the diagonal action of $G$ on $G' \times X$. 
The quotient $(G' \times^{G_0}X)/F$ exists by \cite{Mum70}, and yields the
desired $G$-bundle $\pi'$. 

To construct $\Phi$, note that any $u \in \Aut^G(X)$ yields an automorphism
$$
\id \times u : G' \times X \longrightarrow G' \times X, \quad
(g',x) \longmapsto (g',u(x))
$$ 
which is $G' \times G$-equivariant, and hence descends to $u' \in \Aut^{G'}(X')$.
Moreover, the assignement $u \mapsto u' =: \Phi(u)$ readily extends to a 
homomorphism of group functors, and hence yields the desired homomorphism
of group schemes.

(ii) Let $\cG' := \Aut^{G'}(X')_{\ant}$. Then $\pi'_*: \cG' \to A$ is surjective 
by Proposition \ref{prop:aa}. Also, the natural morphism
$$
p: X' = G' \times^G X \to G'/f(G)
$$
is invariant under $\cG'$ by Lemma \ref{lem:inv}. Thus, $\cG'$ acts on $X$ 
(the fiber of $p$ at the base point of $G'/f(G)$), and this action lifts 
the $A$-action on itself by translations. Hence $\pi$ is homogeneous.
\end{proof}

\begin{lemma}\label{lem:inv}
Let $Z$ be a variety, $\cH$ its largest anti-affine group of automorphisms,
and $f : Z \to W$ a morphism to a quasi-affine variety. Then $f$ is $\cH$-invariant.
\end{lemma}

\begin{proof}
The action of $\cH$ on $Z$ yields a structure of rational $\cH$-module on the algebra
of regular functions $\cO(Z)$. Since $\cH$ is anti-affine, it must act trivially on
$\cO(Z)$ and hence on the image of the algebra homomorphism
$f^{\#} : \cO(W) \longrightarrow \cO(Z)$.
As the regular functions on $W$ separate points, this means that $f$ is $\cH$-invariant. 
\end{proof}

\begin{remarks}\label{rem:csg}
(i) With the notation of the above proposition, the homogeneity of $\pi'$ does not
always imply that of $\pi$. For instance, if $f$ is constant, then $\pi'$ is the trivial
bundle $G' \times A \to A$, and hence is homogeneous even if $\pi$ is not. For less
trivial instances, where $f$ is faithful, see Examples \ref{ex:ext} and \ref{ex:eagst} (ii).

\smallskip

\noindent
(ii) Recall that a closed subgroup $H$ of an affine algebraic group $G$ is 
\emph{observable in $G$}, if the homogeneous space $G/H$ is quasi-affine; we refer to  
\cite[Theorem 2.1]{Gr97} for further characterizations of observable subgroups. Also,
recall that $H$ is observable in $G$ if and only if so is the radical $R(H)$ (the largest
connected solvable normal subgroup of $H$), see [loc.~cit., Theorem 2.7]. Moreover,
any closed commutative subgroup is observable by [loc.~cit., Corollary 2.9]. As a consequence,
\emph{$H$ is observable in $G$ if $R(H)$ is commutative.}
\end{remarks}

We will encounter quasi-affine homogeneous spaces $G/H$ where $H \subset G$ is a possibly
non-reduced subgroup scheme. This presents no additional difficulty in view of the following
result:

\begin{lemma}\label{lem:obs}
Let $G$ be an algebraic group, $H$ a closed subgroup scheme, and $H_{\red} \subset H$
the associated reduced subgroup scheme (i.e., the largest algebraic subgroup).

\smallskip

\noindent
{\rm (i)} If $G/H$ is quasi-affine, then $H$ contains $G_{\ant}$. In other words, $G$ acts 
on $G/H$ via its affinization $G/G_{\ant}$.

\smallskip

\noindent
{\rm (ii)} $G/H$ is quasi-affine if and only if so is $G_0/(G_0 \cap H)$.

\smallskip

\noindent
{\rm (iii)} $G/H$ is quasi-affine if and only if so is $G/H_{\red}$.
\end{lemma}

\begin{proof}
(i) The anti-affine group $G_{\ant}$ acts trivially on the rational $G$-module 
$\cO(G/H)$, and hence on $G/H$ as the latter is quasi-affine.

(ii) follows from the fact that $G_0$ acts on $G/H$ with finitely many orbits, 
all of them being closed and isomorphic as varieties to $G_0/(G_0 \cap H)$. 
Indeed, for any $g \in G$, the $G_0$-orbit
of $g H \in G/H$ is isomorphic to $G_0/(G_0 \cap \Int(g) H) = G_0/\Int(g)(G_0 \cap H)$,
where $\Int(g): G \to G$ denotes the conjugation by $g$. Moreover, $\Int(g)$ induces an 
isomorphism of varieties $G_0/(G_0 \cap H) \to G_0/\Int(g)(G_0 \cap H)$.

(iii) By (ii) and the equality $(G_0 \cap H)_{\red} = G_0 \cap H_{\red}$, we may assume
that $H$ is connected.

If $G/H$ is quasi-affine, then we may find a finite-dimensional vector space 
$V \subset \cO(G/H)$ such that the associated morphism $G/H \to V^*$ is a locally
closed immersion. Then $V$ is contained in a finite-dimensional $G$-submodule of 
$\cO(G/H)$, and hence we may assume that $V$ itself is $G$-stable. Let 
$R \subset \cO(G/H)$ denote the subalgebra generated by $V$, and $S$ the integral closure 
of $R$ in the field of rational functions $k(G/H_{\red})$. Then $S$ is a $G$-stable, 
finitely generated algebra; thus, we may find a finite-dimensional $G$-submodule 
$W \subset S$ that generates $S$. Moreover, since the fraction field of $R$ is $k(G/H)$, 
we see that $k(G/H_{\red})$ is the fraction field of $S$. Thus, the natural map 
$G/H_{\red} \to W^*$ is birational; since it is $G$-equivariant, it is a locally 
closed immersion. Hence $G/H_{\red}$ is quasi-affine.

Conversely, assume that $G/H_{\red}$ is quasi-affine and choose a finite-dimensional
$G$-submodule $W \subset \cO(G/H_{\red})$ that yields a locally closed immersion
$G/H_{\red}\to W^*$. Since the natural morphism $G/H_{\red} \to G/H$ is purely inseparable,
there exists a positive integer $q$ (a power of the characteristic exponent of $k$) such
that $W^q \subset \cO(G/H)$. Let $R \subset \cO(G/H)$ denote the integral closure of 
the subalgebra generated by $W^q$. Then, as in the above step, $R$ is generated by a
finite-dimensional $G$-submodule $V$, and the corresponding map $G/H \to V^*$ is a 
locally closed immersion.
\end{proof}

\subsection{Associated vector bundles}
\label{subsec:avb}

A vector bundle
\begin{equation}\label{eqn:vec}
p: E \longrightarrow A
\end{equation}
is \emph{homogeneous}, if $E \cong \tau_a^*(E)$ for all $a \in A$. 
As for principal bundles, the isomorphisms of vector bundles
$E \to \tau_a^*(E)$ correspond bijectively to those lifts 
$\widetilde{\tau_a}: E \to E$ of $\tau_a : A \to A$ that are linear on fibers, i.e.,
that commute with the action of the multiplicative group $\bG_m$ by scalar multiplication
on fibers. In other words, \emph{$E$ is homogeneous if and only if each translation
$\tau_a$ lifts to some $\widetilde{\tau_a} \in \Aut^{\bG_m}(E)$} (the equivariant automorphism
group of $E$ viewed as a variety with $\bG_m$-action).

Also, the group functor $S \mapsto \Aut^{\bG_m}_S(E \times S)$ is represented by a group
scheme, locally of finite type, that we still denote by $\Aut^{\bG_m}(E)$. (This may be deduced 
from the analogous statement for principal bundles by using Lemma \ref{lem:homvec} (ii); 
alternatively, one may identify $\Aut^{\bG_m}(E)$ with a closed subfunctor of 
$\Aut(\bP(E \oplus O_A))$, where $\bP(E \oplus O_A)$ denotes the projective variety obtained
by completing $E$ with its ``hyperplane at infinity'' $\bP(E)$). Moreover, we have an
exact sequence of group schemes
$$
\CD
1 @>>> \Aut^{\bG_m}_A(E) @>>> \Aut^{\bG_m}(E) @>{p_*}>> \Aut(A),
\endCD
$$
where $\Aut_A^{\bG_m}(E)$ denotes the group scheme of automorphisms of the vector bundle 
$E$ over $A$, and $p_*$ assigns to each $\bG_m$-automorphism of $E$, the induced automorphism
of $A$ viewed as the categorical quotient $E/\!/\bG_m$. Note that $\Aut_A^{\bG_m}(E)$ is a 
connected affine algebraic group, since it is open in the affine space 
$\End_A^{\bG_m}(E) = H^0(A, End(E))$. Thus, $\Aut^{\bG_m}(E)$ is an algebraic group
(locally of finite type). 

We now relate the equivariant automorphism groups and notions of homogeneity for vector
bundles and principal bundles:

\begin{lemma}\label{lem:homvec}
Let $\pi: X \to A$ be a $G$-bundle, and $\rho: G \to \GL(V)$ a finite-dimensional representation.

\smallskip

\noindent
{\rm (i)} The associated vector bundle
$$
p: E := X \times^G V \longrightarrow A
$$
exists, and there is a natural homomorphism
$$
\Phi: \Aut^G(X) \longrightarrow \Aut^{\bG_m}(E)
$$ 
compatible with $\pi_* : \Aut^G(X) \to \Aut(A)$ and 
$p_* : \Aut^{\bG_m}(E) \to \Aut(A)$.

\smallskip

\noindent
{\rm (ii)} If $G = \GL(V)$ and $\rho = \id$, then $\Phi$ is an isomorphism. 

\smallskip

\noindent
{\rm (iii)} If the principal bundle $\pi$ is homogeneous, then so is the vector bundle
$p$. The converse holds if $G = \GL(V)$ and $\rho = \id$.
\end{lemma}

\begin{proof}
(i) The existence of $p$ follows from faithfully flat descent. The construction of
$\Phi$ is similar to that in the proof of Proposition \ref{prop:ext} (ii).

(ii) follows from the correspondence between vector bundles with fiber $V$ and 
principal bundles under $\GL(V)$. Specifically, we have a commutative diagram
$$
\CD
1 @>>> \Aut_A^{\GL(V)}(X) @>>> \Aut^{\GL(V)}(X) @>{\pi_*}>> \Aut(A) \\
& & @V{\Psi}VV @V{\Phi}VV @V{=}VV \\
1 @>>> \Aut_A^{\bG_m}(E) @>>> \Aut^{\bG_m}(E) @>{p_*}>> \Aut(A), \\
\endCD
$$
where $\Psi: \Aut_A^{\GL(V)}(X) = \Hom^{\GL(V)}(X,\GL(V)) 
\to \Aut_A^{\bG_m}(E) \subset H^0(A, End(E))$ 
associates with each equivariant morphism $f : X \to \GL(V)\subset \End(V)$, 
the same $f$ viewed as a section of the associated bundle $\End(E)$. 
Thus, $\Psi$ is an isomorphism. Moreover, $\pi_*$ and $p_*$ have the same image, 
since for a given $u \in \Aut(A)$, the condition that $u^*(X) \cong X$ as $\GL(V)$-bundles 
is equivalent to $u^*(E) \cong E$ as vector bundles.

(iii) is a direct consequence of (i) and (ii).
\end{proof}

Next, we record the following preliminary result:

\begin{lemma}\label{lem:rep}
Any affine algebraic group $G$ admits a faithful finite-dimensional representation
$\rho: G \to \GL(V)$ such that the homogeneous space $\GL(V)/\rho(G)$ is quasi-affine;
in other words, $\rho(G)$ is observable in $\GL(V)$.
\end{lemma}

\begin{proof}
Choose a faithful finite-dimensional representation 
$$ 
\iota : G \longrightarrow \GL(M).
$$ 
By a theorem of Chevalley, there exist a $\GL(M)$-module $N$ and a line 
$\ell \subset N$ such that $G$ is the isotropy subgroup scheme of $\ell$ in
$\GL(M)$. Then $G$ acts on $\ell$ via a character $\chi : G \to \bG_m$. This yields
a faithful homomorphism
$$
(\iota, \chi) : G \longrightarrow \GL(M) \times \bG_m.
$$
Moreover, $\GL(M) \times \bG_m$ acts linearly on $N$ via 
$$
(g,t) \cdot x := t^{-1} g \cdot x,
$$
and the isotropy subgroup scheme of any point $x \in \ell$ for this action is exactly
$(\iota \times \chi)(G)$. It follows that the homogeneous space
$$
(\GL(M) \times \bG_m)/(\iota, \chi)(G) \cong (\GL(M) \times \bG_m) \cdot x
$$
is a locally closed subvariety of $N$, and hence is quasi-affine. Next, the natural
embedding
$$
\GL(M) \times \bG_m = \GL(M) \times \GL(1) \subset \GL(M \oplus k)
$$
yields a faithful representation $\rho: G \to \GL(M \oplus k)$, and the homogeneous
space
$$
\GL(M\oplus k)/\rho(G) = \GL(M \oplus k) \times^{\GL(M) \times \bG_m}
(\GL(M) \times \bG_m)/(\iota, \chi)(G)
$$
is a locally closed subvariety of $\GL(M \oplus k)\times^{\GL(M) \times \bG_m} N$.
The latter is the total space of a vector bundle over the affine variety
$\GL(M \oplus k)/(\GL(M) \times \bG_m)$, and hence is affine as well. Thus, 
$\GL(M \oplus k)/\rho(G)$ is quasi-affine. 
\end{proof}

We may now state the main result of this subsection:

\begin{theorem}\label{thm:ass}
Given a $G$-bundle (\ref{eqn:prin}) where $G$ is affine, the following conditions are
equivalent:

\smallskip

\noindent
{\rm (i)} $\pi$ is homogeneous.

\smallskip

\noindent
{\rm (ii)} All vector bundles associated with irreducible representations
are homogeneous. 

\smallskip

\noindent
{\rm (iii)} All associated vector bundles are homogeneous.

\smallskip

\noindent
{\rm (iv)} The associated vector bundle (\ref{eqn:vec}) is homogeneous for some
faithful finite-dimensional representation $\rho: G \to \GL(V)$ such that $\rho(G)$
is observable in $\GL(V)$. (Such a representation exists by Lemma \ref{lem:rep}).
 
\end{theorem}

\begin{proof}
(i)$\Rightarrow$(ii) follows from Lemma \ref{lem:homvec} (iii).

(ii)$\Rightarrow$(iii) holds since any extension of homogeneous vector bundles
is homogeneous, as follows from \cite[Theorem 4.19]{Muk79} that characterizes
homogeneity in terms of the Fourier-Mukai transform. 

We sketch an alternative proof of this fact via methods of algebraic groups: 
given two homogeneous vector bundles $E_1$, $E_2$ over $A$ with respective anti-affine 
groups of automorphisms $\cG_1$, $\cG_2$, the fibered product 
$\cG_1 \times_A \cG_2$ acts linearly on $\Ext^1_A(E_2,E_1)$ and its largest
anti-affine subgroup $\cG$ acts trivially. Thus, for any extension
$0 \to E_1  \to E \to E_2 \to 0$, we see that $\cG$ acts on $E$ by lifting the 
translations in $A$.   

(iii)$\Rightarrow$(iv) is obvious.

(iv)$\Rightarrow$(i) The $\GL(V)$-bundle $\GL(V) \times^G X \to A$ is homogeneous 
by Lemma \ref{lem:homvec} (iii). Since $\GL(V)/\rho(G)$ is quasi-affine, it follows
that $\pi$ is homogeneous, in view of Proposition \ref{prop:ext} (ii).
\end{proof}

As a direct consequence, we obtain:

\begin{corollary}\label{cor:red}
Let $V$ be a finite-dimensional vector space, and $G \subset \GL(V)$ a reductive subgroup.
Then a $G$-bundle $\pi: X \to A$ is homogeneous if and only if the associated vector bundle 
$p: X \times^G V \to A$ is homogeneous.
\end{corollary}

(Indeed, the variety $\GL(V)/G$ is affine). 

Note that the same statement holds, with a completely different proof, for $G$-bundles 
over complete rational homogeneous varieties in characteristic $0$; see 
\cite[Theorem 1.1]{BT10}.

Another straightforward consequence of Theorem \ref{thm:ass} is the following result: 

\begin{corollary}\label{cor:uni}
Let $G$ be an affine algebraic group, $U \subset G$ a normal unipotent subgroup scheme, 
$\pi: X \to A$ a $G$-bundle, and $\bar{\pi} : \bar{X} := X/U \to A$
the associated principal bundle under $\bar{G}:= G/U$.
Then $\pi$ is homogeneous if and only if so is $\bar{\pi}$.

In particular, any principal bundle under a unipotent group is homogeneous.
\end{corollary}

(Indeed, any irreducible representation $\rho: G \to \GL(V)$ factors through $\bar{G}$).

Combining Corollary \ref{cor:uni} with Proposition \ref{prop:conn}, Theorem \ref{thm:conn}
and Theorem \ref{thm:aff}, we obtain the equivalence (a)$\Leftrightarrow$(c) in Theorem 
\ref{thm:main}. The equivalences (a)$\Leftrightarrow$(d)$\Leftrightarrow$(e)
will be obtained in Subsection \ref{subsec:con}, and (a)$\Leftrightarrow$(b) in Subsection
\ref{subsec:hbhs}. 

To conclude this subsection, we provide the example announced in Remark \ref{rem:csg} (i):

\begin{example}\label{ex:ext}
Consider a non-trivial line bundle $L$ over $A$, generated by $n$ global sections. 
Then $L$ fits into an exact sequence of vector bundles over $A$
$$
0 \longrightarrow E \longrightarrow O_A^n \longrightarrow L  \longrightarrow 0
$$
where $O_A$ denotes the trivial line bundle. This corresponds to a reduction of the trivial 
$\GL_n$-bundle $\pi' : \GL_n \times A \to A$ 
to a principal $P$-bundle $\pi: X \to A$, where $P \subset \GL_n$ denotes the parabolic subgroup
that stabilizes a hyperplane. Taking for $U$ the unipotent radical of $P$, we see that 
$\bar{P} \cong \GL_{n-1} \times \bG_m$ and the associated bundle $\bar{\pi} : \bar{X} \to A$ 
is the product over $A$ of the principal bundles associated with $E$ and $L$. Since $L$ is not
homogeneous, the same holds for $\bar{\pi}$, and hence $\pi$ is not homogeneous as well; but 
$\pi'$ is obviously homogeneous.   
\end{example}

\subsection{Connections}
\label{subsec:con}

We first recall some classical constructions associated with principal bundles
(see \cite{At57b}) in our special setting. For a $G$-bundle $\pi: X \to A$, 
we have an exact sequence of tangent sheaves on $X$
$$
\CD
0 @>>> T_{\pi} @>>> T_X @>{d\pi}>> \pi^* T_A @>>> 0 
\endCD
$$
and isomorphisms of sheaves of Lie algebras
$$
\CD \cO_X \otimes \Lie(G) @>{\cong}>> T_{\pi} \endCD, \quad 
\CD \cO_A \otimes \Lie(A) @>{\cong}>> T_A \endCD
$$
with an obvious notation. This yields an exact sequence of sheaves of Lie algebras on $X$
\begin{equation}\label{eqn:tan}
\CD
0 @>>> \cO_X \otimes \Lie(G) @>>> T_X @>{d\pi}>> \cO_X \otimes \Lie(A) @>>> 0.
\endCD
\end{equation}
Also, all these sheaves are $G$-linearized, where $G$ acts on $\Lie(G)$ via its adjoint 
representation, and $G$ acts trivially on $\Lie(A)$; moreover, the morphisms in 
(\ref{eqn:tan}) preserve the linearizations. 

Now the category of $G$-linearized locally free sheaves on $X$ is equivalent to that of 
locally free sheaves on $A$, via the invariant direct image $\pi_*^G$ and the pull-back $\pi^*$. 
Thus, (\ref{eqn:tan}) corresponds to an exact sequence of sheaves of Lie algebras on $A$,
the \emph{Atiyah sequence}
\begin{equation}\label{eqn:ati}
\CD
0 @>>> \pi_*\big( \cO_X \otimes \Lie(G) \big)^G @>>> \At(X) @>{d\pi}>> 
\cO_A \otimes \Lie(A) @>>> 0
\endCD
\end{equation}
where the left-hand side is the sheaf of local sections of the associated vector bundle
$X \times^G \Lie(G) \to A$ (the \emph{adjoint bundle}), and $\At(X) := \pi_*(T_X)^G$.

Taking global sections in (\ref{eqn:ati}) yields an exact sequence of Lie algebras
\begin{equation}\label{eqn:lie}
\CD
0 @>>> \Hom^G\big( X, \Lie(G) \big) @>>> H^0(X,T_X)^G @>{d\pi}>> \Lie(A)
\endCD
\end{equation}
which may also be obtained from the exact sequence of group schemes (\ref{eqn:sch}) by taking
their Lie algebras (see e.g. \cite[Section 4]{Br10b}).

Moreover, the $G$-equivariant splittings of the exact sequence (\ref{eqn:tan}) correspond to 
the splittings of (\ref{eqn:ati}), i.e., to the sections of 
$d\pi : H^0(X,T_X)^G \to \Lie(A)$. Such a section
$$
D : \Lie(A) \longrightarrow H^0(X,T_X)^G
$$ 
is called a \emph{connection}; $D$ is \emph{integrable} (or \emph{flat}) if it is a homomorphism 
of Lie algebras, i.e., its image is an abelian Lie subalgebra.

We may now state the following algebraic version of a homogeneity criterion for holomorphic
bundles over complex tori, due to Matsushima (see \cite[Th\'eor\`eme 1]{Ma59}) and Morimoto 
(\cite[Th\'eor\`eme 1]{Mo59}):

\begin{proposition}\label{prop:con}
In characteristic $0$, the following conditions are equivalent for a $G$-bundle $\pi : X \to A$:

\smallskip

\noindent
{\rm (i)} $\pi$ is homogeneous.

\smallskip

\noindent
{\rm (ii)} $\pi$ admits an integrable connection. 

\smallskip

\noindent
{\rm (iii)} $\pi$ admits a connection.
\end{proposition}  

\begin{proof}
(i)$\Rightarrow$(ii) The restriction $\pi_{* \vert \cG}: \cG \to A$ is surjective 
by Proposition \ref{prop:aa}, and hence so is 
$d\pi_{*\vert \cG} : \Lie \big( \cG \big) \to \Lie(A)$. 
Any splitting of the latter map yields an integrable connection, since $\cG$ is commutative.

(ii)$\Rightarrow$(iii) is left to the reader.

(iii)$\Rightarrow$(i) The map $d\pi : H^0(X,T_X)^G \to \Lie(A)$ is surjective, and hence  
$\pi_*: \Aut^G(X) \to \Aut(A)$ is surjective on neutral components.
\end{proof}

When $k = \bC$ and $G$ is a connected reductive group, the above statement
has also been proved by Biswas (see \cite[Theorem 3.1]{Bi04}), via very different arguments
based on semistability. The case of holomorphic principal bundles over
complex tori is treated by Biswas and Gomez (see \cite[Theorem 3.1]{BG08}) via the 
theory of Higgs $G$-bundles; in this setting, they also characterize homogeneous bundles
in terms of pseudostability and vanishing of characteristic classes. 

On the other hand, in positive characteristics, the group scheme $\Aut^G(X)$ 
need not be reduced (see Example \ref{ex:eagst} (iii)) and hence characterizing 
homogeneous bundles in terms of connections would make little sense.

\section{Structure}
\label{sec:st}

\subsection{Homogeneous bundles as homogeneous spaces}
\label{subsec:hbhs}

By a refinement of Proposition \ref{prop:aa}, we obtain our main structure result:

\begin{theorem}\label{thm:str}
Let $\pi : X \to A$ be a principal bundle under an algebraic group $G$ and consider
the $G \times \cG$-action on $X$, where $\cG \subset \Aut^G(X)$ denotes the largest anti-affine 
subgroup. Then the following conditions are equivalent:

\smallskip

\noindent
{\rm (i)} $\pi$ is homogeneous.

\smallskip

\noindent
{\rm (ii)} There is a $G \times \cG$-equivariant isomorphism
\begin{equation}\label{eqn:str}
X \cong G \times^H \cG,
\end{equation}
where $H \subset G$ is a closed subgroup scheme that also fits into an extension
of commutative group schemes
\begin{equation}\label{eqn:ext}
\CD
0 @>>> H @>>> \cG @>{\pi_*}>> A @>>> 0
\endCD
\end{equation}
via $\pi_* : \cG \to A$. (In particular, $H$ is commutative).

\smallskip

Moreover, the assignement $X \mapsto (\cG,H)$ yields a bijection between the 
isomorphism classes of $G$-bundles over $A$, and the equivalence classes of 
pairs consisting of a commutative subgroup scheme $H \subset G$
and an extension (\ref{eqn:ext}), where $\cG$ is anti-affine. 
Here two pairs are equivalent if the corresponding subgroups $H$, $H'$
are conjugate in $G$, and if in addition the resulting isomorphism 
$H \to H'$ lifts to an isomorphism of extensions. 
\end{theorem}

\begin{proof}
Assume that $\pi$ is homogeneous and consider the action of $G \times \cG$
on $X$. This action is transitive, since $G$ acts transitively on the fibers
of $\pi$ while $\cG$ acts transitively on $A$ (Proposition \ref{prop:aa}). 
Thus, $X \cong (G \times \cG)/H$ as $G \times \cG$-varieties, where $H \subset G \times \cG$ 
denotes the isotropy subgroup scheme of some point $x \in X$; this
isomorphism maps $x$ to the base point $(e_G,e_{\cG})H$ of 
the homogeneous space $(G \times \cG)/H$.
Denote by $p: H \to G$ and $q: H \to \cG$ the projections. Then the (scheme-theoretic)
kernel of $q$ is identified with a subgroup scheme of $G$ that fixes $x$. 
Since $\pi$ is a $G$-bundle, it follows that $\Ker(q)$ is trivial. Likewise,
the kernel of $p$ is identified with a subgroup scheme of $\cG$ that fixes $x$. 
Since $\cG$ centralizes $G \times \cG$, it follows that $\Ker(p)$ acts trivially 
on $X$, and hence is trivial as $\cG$ acts faithfully. This yields the assertions of (ii).

Conversely, if (ii) holds, then $\cG$ acts on $X$ by $G$-automorphisms that lift 
the translations of $A$. Thus, $\pi$ is homogeneous.

This proves the equivalence of (i) and (ii). The final assertion follows from the
fact that the isotropy subgroup scheme $H = (G \times \cG)_x$ is uniquely determined
up to conjugacy in $G \times \cG$ and hence in $G$.
\end{proof}

\begin{corollary}\label{cor:str}
Given an abelian variety $A$ and a connected commutative algebraic group $G$, the isomorphism 
classes of homogeneous $G$-bundles over $A$ correspond bijectively to the isomorphism classes 
of extensions of connected algebraic groups 
\begin{equation}\label{eqn:com}
1 \longrightarrow G \longrightarrow X \longrightarrow A \longrightarrow 1.
\end{equation}
Moreover, any such extension is commutative.
\end{corollary}

\begin{proof}
For any homogeneous $G$-bundle (\ref{eqn:prin}), the variety $X := (G \times \cG)/H$ 
has a structure of commutative algebraic group such that the projection $\pi: X \to A$
is a homomorphism with kernel $G$. This group structure is unique up to the choice of the 
neutral element in the fiber of $\pi$ at $0_A$, i.e., up to a translation in $G$.

Conversely, any extension (\ref{eqn:com}) obviously yields a homogeneous 
$G$-bundle. Moreover, $X = X_{\aff} X_{\ant}$, and $X_{\aff} \subset G$ is commutative. 
Since $X_{\ant}$ is central in $X$, it follows that $X$ is commutative.

Also, note that the automorphism group of a $G$-bundle $\pi: X \to A$ is
$$
\Aut^G_A(X) \cong \Hom^G(X,G) = \Hom(A,G) 
$$
as seen in Remark \ref{rem:red}, and the automorphism group of an extension (\ref{eqn:com}) is
$\Hom_{\gp}(A,G)$. Moreover, $\Hom(A,G)$ is the semi-direct product of $\Hom_{\gp}(A,G)$ with 
the normal subgroup $G$ of translations. Putting these facts together yields our assertion.
\end{proof}

\begin{remarks}\label{rem:hbhs}
(i) If $k$ is the algebraic closure of a finite field, then every anti-affine group 
is an abelian variety (see \cite[Proposition 2.2]{Br09}). Thus, $\pi_* : \cG \to A$ 
is an isogeny for any homogeneous $G$-bundle (\ref{eqn:prin}). In other words, 
the group scheme $H$ is finite. This does not extend to any other algebraically closed 
field, in view of \cite[Example 3.11]{Br09}.

\smallskip

\noindent
(ii) With the notation of Theorem \ref{thm:str}, the subgroup scheme $H$ always contains
$G_{\ant}$, since by Remark \ref{rem:red},
$$
G_{\ant} = \Aut^G_A(X)_{\ant} \subset \Aut^G(X)_{\ant} \cap \Aut_A(X) = \Ker(\pi_*: \cG \to A).
$$
Moreover, we have for the trivial bundle: $H = G_{\ant}$ and $\cG = G_{\ant} \times A$. 
So Theorem \ref{thm:str} (stated for an arbitrary algebraic group $G$) yields a
``natural'' classification if and only if $G$ is \emph{affine}; then the trivial bundle does
correspond to the trivial subgroup scheme $H$. 

\smallskip

\noindent
(iii) With the notation of Theorem \ref{thm:str} again, the projection 
$G \times^H \cG \to G/H$ yields a $G$-equivariant morphism
$$
\varphi : X \longrightarrow G/H
$$
which is a principal bundle under the anti-affine group $\cG$. We may view $\varphi$ 
as a reduction of structure group to the $H$-bundle $\pi_* : \cG \to A$.
If $G$ is affine, then the variety $G/H$ is quasi-affine, as follows from
\cite[Corollary 1.2.9]{Gr97} together with Lemma \ref{lem:obs}. Moreover, the morphism 
$\varphi^{\#}: \cO(G/H) \to \cO(X)$ 
is an isomorphism, since $\cG$ is anti-affine. Thus, $\varphi$ may be viewed as the canonical 
morphism $X \to \Spec \, \cO(X)$ and hence depends only on the abstract variety $X$. 

If in addition $G$ is connected, then $\pi : X \to A$ is the Albanese morphism 
(the universal morphism to an abelian variety). Thus, $\pi$ also depends only 
on the abstract variety $X$. 
\end{remarks}

In fact, the above reduction of structure group is minimal in the following sense: 

\begin{proposition}\label{prop:min}
Let $\pi: X \to A$ be a homogeneous principal bundle under an affine algebraic 
group $G$ and let $H \subset G$ be a subgroup scheme as in Theorem \ref{thm:str}.
Then $\pi$ admits a reduction to an observable subgroup scheme $K \subset G$
if and only if $K$ contains a conjugate of $H$ in $G$. 
\end{proposition}

\begin{proof}
Consider a reduction of structure group $q : X \to G/K$. If $K$ is observable in $G$, 
then $q$ is $\cG$-invariant by Lemma \ref{lem:inv} and hence is $(G \times \cG)$-equivariant 
for the trivial action of $\cG$ on $G/K$. Thus, $q$ maps the base point of 
$X = (G\times \cG)/H$ to an $H$-fixed point. In other words, $H$ is contained in 
a conjugate of $K$. The converse implication is obvious.
\end{proof}

In the above statement, the assumption that $K$ is observable cannot be omitted,
as shown by Example \ref{ex:eagst} (ii).

\subsection{Equivariant automorphism groups}
\label{subsec:eagst}

The structure theorem \ref{thm:str} yields a further characterization of homogeneity:

\begin{proposition}\label{prop:com}
The following conditions are equivalent for a principal bundle $\pi: X \to A$
under an affine algebraic group $G$:

\smallskip

\noindent
{\rm (i)} $\pi$ is homogeneous.

\smallskip

\noindent
{\rm (ii)} $\pi$ is trivialized by an anti-affine extension $\varphi : \cG \to A$.
\end{proposition}

\begin{proof}
(i)$\Rightarrow$(ii) Theorem \ref{thm:str} yields a cartesian square
$$
\CD
G \times \cG @>{p}>> \cG \\
@V{q}VV @V{\pi_*}VV \\
X @>{\pi}>> A, \\
\endCD
$$
where $p$ denotes the second projection, and $q$ the quotient morphism by $H$.

(ii)$\Rightarrow$(i) By assumption, we have a similar cartesian square with $\pi_*$
replaced by $\varphi$. Thus, the group scheme $H := \Ker(\varphi)$ acts on 
$G \times \cG$ by $G$-equivariant automorphisms that lift the $H$-action on $\cG$ by
multiplication: this action must be of the form
$$
h \cdot (g,\gamma) = \big( g f(h,\gamma), h \gamma \big)
$$
for some morphism $f: H \times \cG \to G$. Since $\cG$ is anti-affine and $G$ is affine,
$f$ factors through a morphism $\psi : H \to G$. Moreover, $\psi$ is a homomorphism of
group schemes, since $H$ acts on $G \times \cG$. Thus, $X \cong G \times^H \cG$ 
is homogeneous.
\end{proof}

\begin{remark}\label{rem:eagst}
In particular, a $G$-bundle is homogeneous if it is trivialized by an isogeny,
or equivalently, by the multiplication map $n_A : A \to A$ for some positive
integer $n$. 

Also, given a homogeneous bundle $X \cong G \times^H \cG$, one easily checks that
the pull-back $n_A^*(X)$ is the $G$-bundle $G \times^{H/H_n} \cG/H_n$, where
$H_n \subset H$ denotes the kernel of $n_H$, and $H/H_n$ is viewed as a subgroup 
of $G$ via the isomorphism $n_H: H/H_n \to H$. In particular, 
\emph{$X$ is trivialized by $n_A$ if and only if $H$ is killed by $n$}.     
\end{remark}

Theorem \ref{thm:str} also yields a description of bundle automorphisms:

\begin{proposition}\label{prop:aut}
Consider a homogeneous bundle $\pi: X \to A$ under an affine algebraic group $G$.
Write $X = G \times^H \cG$ as in Theorem \ref{thm:str} and denote by $C_G(H)$ 
the centralizer of $H$ in $G$. Then the group scheme $C_G(H)$ acts on $X$ by bundle 
automorphisms, via $z \cdot ( g, \gamma ) H := ( gz^{-1},\gamma )H$. 
Moreover, this action yields an isomorphism
\begin{equation}\label{eqn:cent}
\CD
F : C_G(H) @>{\cong}>> \Aut^G_A(X).
\endCD
\end{equation}
As a consequence, the group scheme $\pi_*^{-1}(A)$ (consisting of those $G$-automorphisms 
of $X$ that lift the translations of $A$) is isomorphic to 
$C_G(H) \times ^H \cG = \big( C_G(H) \times \cG \big)/H$.
\end{proposition}

\begin{proof}
Given a scheme $S$ and a morphism $\psi : S \to C_G(H)$, we obtain a morphism
$$
G \times \cG \times S \to G \times \cG \times S, 
\quad (g, \gamma, s) \mapsto \big( g \psi(s)^{-1}, \gamma,s)
$$
which is $G \times H$-equivariant and hence descends to a $G$-equivariant morphism
$F(\psi) : X \times S \to X \times S$. Clearly, $F(\psi)$ is compatible with 
$\pi \times \id : X \times S \to A \times S$; in other words, 
$F(\psi) \in \Aut_A^G(X)(S)$. This yields a homomorphism of group functors 
$F : C_G(H) \to \Aut_A(X)$.
 
To show that $F$ is an isomorphism, note that 
$$
\Aut^G_A(X)(S) = \Aut^G_{A \times S}(X \times S) \cong \Hom^G(G \times^H \cG \times S,G)
\cong \Hom^H(\cG \times S,G)
$$
where $G$, and hence $H$, acts trivially on $S$. Now each $f \in \Hom(\cG \times S,G)$
factors through a morphism $\varphi : S \to G$, as in the proof of Proposition \ref{prop:com}. 
Moreover, $f$ is $H$-equivariant if and only if $\varphi$ is $H$-equivariant, 
i.e., $\varphi$ factors through a morphism $\psi: S \to C_G(H)$. Tracing through the 
above isomorphisms, one checks that $F$ and the assignement $f \mapsto \psi$ are mutually 
inverse.

For the final assertion, note that $\pi_*^{-1}(A) =  \Aut_A^G(X) \, \cG = C_G(H) \, \cG$. 
Moreover, the scheme-theoretic intersection $\Aut_A^G(X) \cap \cG$ is the kernel of 
$\pi_* : \cG \to A$, that is, $H$.
\end{proof}

\begin{examples}\label{ex:eagst}
(i) If $G$ is the general linear group $\GL(V)$, then $C_G(H)$ is the group of invertible elements
of a finite-dimensional associative algebra: the centralizer of $H$ in $\End(V)$. In particular, 
$C_G(H)$ is a connected algebraic group, and hence $\Aut^G(X)$ is an algebraic group (locally of
finite type). This also follows from the well-known structure of automorphisms of vector bundles, 
in view of Lemma \ref{lem:homvec}. 

\smallskip

\noindent
{\rm (ii)} If $G = \SO_3$ is the special orthogonal group associated to the quadratic form 
$x^2 + y^2 + z^2$ in characteristic $\neq 2$, then the subgroup $H \subset G$ consisting 
of diagonal matrices is isomorphic to $\bZ/2\bZ \times \bZ/2\bZ$ and satisfies $H = C_G(H)$. 
Let $\cG$ be an elliptic curve; then the kernel of the multiplication map $2_{\cG}$ is isomorphic
to $H$. Then $X := G \times^H \cG$ is a homogeneous $G$-bundle over the
elliptic curve $A := \cG/H$, and $\Aut^G_A(X)\cong H$ is not connected. 

In fact, $\Aut^G(X)$ is not connected either: since any $h \in H$ satisfies $h = h^{-1}$, 
the $G$-automorphism of $G \times \cG$ given by $(g,\gamma) \mapsto (g, -\gamma)$ descends 
to a $G$-automorphism of $X$ that lifts $- \id_A \in \Aut_{\gp}(A)$.

Next, we claim that $\pi$ admits a reduction of structure group to a principal
bundle under a Borel subgroup $B \subset G$, but no such reduction is homogeneous.

Indeed, $\pi$ is locally trivial for the Zariski topology in view of \cite[Section 5.6]{Se58}.
Equivalently, there exists a $H$-equivariant rational map $\sigma : \cG - \to G$. 
Composing with the projection $G \to G/B$ yields an $H$-equivariant rational map 
$\tau : \cG - \to G/B$ which is in fact regular, since $\cG$ is a smooth complete curve. 
This yields in turn a $G$-equivariant morphism $f : X= G \times^H \cG \to G/B$, i.e., 
a reduction of structure group to $B$.

Assume that such a reduction is homogeneous. Then we have a $G$-isomorphism
$X \cong G \times^B X'$ where $X' \subset X$ is a closed $B$-stable subvariety, and
$\pi: X' \to A$ is a $B$-bundle; moreover, 
$Y \cong B \times^{H'} \cG'$ where $\cG'$ is an anti-affine extension of $A$ by $H'$, 
and $H' \subset B$ is a commutative subgroup. Therefore, $X \cong G \times^{H'} \cG'$;
moreover, the resulting projection $\varphi' : X \to G/H'$ is the canonical morphism
$X \to \Spec \, \cO(X)$, by Remark \ref{rem:hbhs} (iii). In particular, the $G$-equivariant morphism 
$\varphi'$ is the same as $\varphi : X \to G/H$. Thus, $H$ is conjugate to a subgroup of $B$. 
As a consequence, $H$ stabilizes a line in $k^3$ that is isotropic for the quadratic form 
$x^2 + y^2 + z^2$. But this contradicts the fact that the $H$-stable lines are exactly 
the coordinate lines.

\smallskip

\noindent
{\rm (iii)} If $G = \SL_2$ in characteristic $2$ and $H$ is the subgroup generated 
by a non-trivial unipotent element, then $H \cong \bZ/2\bZ$ and 
$C_G(H) \cong \mu_2 \times \bG_a$, where $\mu_2$ is the scheme-theoretic center of $\SL_2$, 
and $\bG_a$ the subgroup of upper triangular unipotent matrices.
In particular, $H$ is reduced whereas $C_G(H)$ is not. Let $\cG$ be an ordinary elliptic curve,
so that $H$ is the set-theoretic kernel of the multiplication map $2_{\cG}$. Then again, 
$X := G \times^H \cG$ is a homogeneous bundle over the elliptic curve $A := \cG/H$; but
now $\Aut^G_A(X)$ is not reduced, nor is $\Aut^G(X)$.
\end{examples}

\subsection{Anti-affine extensions}
\label{subsec:aae}

By Theorem \ref{thm:str}, the classification of homogeneous $G$-bundles over $A$ 
reduces to those of commutative subgroup schemes $H \subset G$ up to conjugacy, 
and of anti-affine extensions $\cG$ of $A$ by $H$. We now determine these extensions
under the assumption that $H$ is affine; this holds of course if $G$ is affine. 
For this, we adapt the arguments of \cite[Section 2]{Br09}, which treats the case that 
$H$ is a connected (affine) algebraic group over an arbitrary field.
 
Given an abelian variety $A$ and a commutative affine group scheme $H$, consider an extension 
(\ref{eqn:ext}) where $\cG$ is a connected commutative algebraic group. We have
$$
H = H_u \times H_s
$$
where $H_u$ (resp. $H_s$) denotes the largest unipotent (resp. diagonalizable) subgroup
scheme of $H$. Thus, setting $\cG_s := \cG/H_u$ and $\cG_u := \cG/H_s$, we obtain extensions
\begin{equation}\label{eqn:exts}
\CD
0 @>>> H_s @>>> \cG_s @>{\pi_s}>> A @>>> 0,
\endCD
\end{equation}
\begin{equation}\label{eqn:extu}
\CD
0 @>>> H_u @>>> \cG_u @>{\pi_u}>> A @>>> 0,
\endCD
\end{equation}
and an isomorphism 
$$
\cG \cong \cG_u \times_A \cG_s
$$ 
arising from the natural homomorphisms $\cG \to \cG_u$, $\cG \to \cG_s$ over $A$.  
We may now state the following result, that generalizes \cite[Proposition 2.5]{Br09} 
with a simpler proof:

\begin{lemma}\label{lem:prod}
With the above notation, $\cG$ is anti-affine if and only if so are $\cG_u$ and $\cG_s$.
\end{lemma}

\begin{proof}
If $\cG$ is anti-affine, then so are its quotient groups $\cG_u$ and $\cG_s$. Conversely, assume
that $\cG_u$ and $\cG_s$ are anti-affine and consider a surjective homomorphism $f : \cG \to G$
where $G$ denotes a (connected, commutative) affine algebraic group. Then $f$ induces a surjective 
homomorphism $f_u : \cG_u = \cG/H_s \to G/G_s$ which must be trivial, since $\cG_u$ is anti-affine 
and $G/G_s$ is affine. Thus, $G = G_s$ is diagonalizable. Likewise, $G$ is unipotent, and hence is 
trivial. This shows that any affine quotient of $\cG$ is trivial, i.e., $\cG$ is anti-affine.
\end{proof}

By that lemma, we may treat the unipotent and diagonalizable cases separately; we begin with 
the former case. In characteristic $0$, the commutative unipotent group $H_u$ is a 
\emph{vector group}, that is, the additive group of a finite-dimensional vector space. 
Moreover, there is a universal extension of $A$ by a vector group,
$$
0 \longrightarrow H(A) \longrightarrow E(A) \longrightarrow A \longrightarrow 0,
$$
where we set
$$
H(A) := H^1(A,\cO_A)^*.
$$ 
Also, recall that $H^1(A,\cO_A)$ is the Lie algebra of the dual abelian variety 
$\widehat{A}$; in particular, $\dim H(A) = \dim(A)$.
Thus, the extension (\ref{eqn:extu}) is classified by a linear map 
$$
\gamma :  H(A) \longrightarrow H_u.
$$
By \cite[Proposition 2.3]{Br09}, $\cG_u$ is anti-affine if and only if $\gamma$ is surjective.

The situation in characteristic $p >0$ is quite different:

\begin{lemma}\label{lem:uni}
In positive characteristics, the extension (\ref{eqn:extu}) is anti-affine if and only if $\cG_u$ 
is an abelian variety. Equivalently, $H_u$ is finite and its Cartier dual $\widehat{H_u}$ 
is a local subgroup scheme of $\widehat{A}$, the kernel of the isogeny 
$\widehat{A} \to \widehat{\cG_u}$ dual of the isogeny $\pi_u : \cG_u \to A$.
\end{lemma}

\begin{proof}
If $\cG_u$ is anti-affine, then it is a semi-abelian variety by \cite[Proposition 2.2]{Br09}.
In other words, $(\cG_u)_{\aff}$ is a torus. But $(\cG_u)_{\aff}$ is contained in $H_u$, and hence 
is trivial: $\cG_u$ is an abelian variety. The remaining assertion follows by duality 
for isogenies of abelian varieties.
\end{proof} 

Next, we treat the diagonalizable case in arbitrary characteristic. Denote by $\widehat{H_s}$ 
the character group of $H_s$. Then any $\chi \in \widehat{H_s}$ yields an extension
$$
0 \longrightarrow \bG_m \longrightarrow \cG_s \times^{H_s} \bG_m \longrightarrow A \longrightarrow 0,
$$
the pushout of the extension (\ref{eqn:exts}) by $\chi : H_s \to \bG_m$. The line bundle
on $A$ associated with this $\bG_m$-bundle on $A$ is algebraically trivial (see e.g. 
\cite[VII Th\'eor\`eme 6]{Se59}) and hence yields a point $c(\chi) \in \widehat{A}(k)$. 
The so defined  map
$$
c : \widehat{H_s} \longrightarrow \widehat{A}(k)
$$
is easily seen to be a homomorphism of abstract groups, that classifies the extension 
(\ref{eqn:exts}).
  
\begin{lemma}\label{lem:dia}
With the above notation and assumptions, $\cG_s$ is anti-affine if and only if $c$ is injective.
\end{lemma}

\begin{proof}
This statement is exactly \cite[Proposition 2.1 (i)]{Br09} in the case that $H_s$ is a torus. 
The general case reduces to that one as follows: $H_s$ is an extension of a finite 
diagonalizable group scheme $F$ by a torus $T$, which yields an exact sequence of character groups
$$
0  \longrightarrow \widehat{F} \longrightarrow \widehat{H_s} \longrightarrow \widehat{T}  
\longrightarrow 0. 
$$
Also, the quotient $B := \cG_s/T$ is an abelian variety, and hence $\cG_s$ is a semi-abelian 
variety. Moreover, we have an exact sequence
$$
0 \longrightarrow F \longrightarrow B \longrightarrow A \longrightarrow 0,
$$
and the dual exact sequence fits into a commutative diagram 
$$
\CD
0  @>>> & \widehat{F}  @>>> & \widehat{H_s}  @>>> & \widehat{T}  @>>> 0 \\
& & & @V{=}VV & @V{c}VV & @V{c_0}VV \\
0  @>>> & \widehat{F}  @>>> & \widehat{A}  @>>> & \widehat{B} @>>> 0,  \\ 
\endCD
$$
where $c_0$ is the classifying homomorphism for the extension 
$0 \to T \to \cG_s \to B \to 0$. Now $\cG_s$ is anti-affine if and only if $c_0$ is injective;
this is equivalent to $c$ being injective. 
\end{proof}

Combining the above results yields the desired classification:

\begin{theorem}\label{thm:cla}
In characteristic $0$, the anti-affine extensions of $A$ by $H$ are classified 
by the pairs $(\delta, c)$, where $\delta : H_u^* \to H^1(A,\cO_A)$ is an injective linear map,
and $c: \widehat{H_s} \to \widehat{A}(k)$ is an injective homomorphism.

In positive characteristics, the anti-affine extensions of $A$ by $H$ exist
only if $H_u$ is finite; then they are classified by the pairs $(\delta, c)$, where 
$\delta: \widehat{H_u} \to \widehat{A}$ is a faithful homomorphism of group schemes, and 
$c: \widehat{H_s} \to \widehat{A}(k)$ is an injective homomorphism.
\end{theorem}

\begin{remarks}
(i) The above theorem yields a description of the image of the homomorphism
$\pi_* : \Aut^G(X) \to \Aut(A)$
for a homogeneous $G$-bundle $\pi: X \to A$. Indeed, since that image contains $A$, it suffices 
to determine its intersection with $\Aut_{\gp}(A)$. Let $X$ be given by a commutative 
subgroup scheme $H \subset G$ and data $\delta,c$ as above. Then $\varphi \in \Aut_{\gp}(A)$
admits a lift in $\Aut^G(X)$ if and only if there exists $z \in N_G(H)$ (the normalizer of
$H$ in $G$) such that the square 
$$
\CD
\widehat{H_s} @>{c}>> \widehat{A} \\
@V{\Int(z)}VV @V{\widehat{\varphi}}VV \\
\widehat{H_s} @>{c}>> \widehat{A} \\
\endCD
$$  
commutes, and the analogous square with $\delta$ commutes as well. Here $\Int(z)$ denotes 
the conjugation by $z$, that acts on $H$ and hence on $H_s$, $H_u$ and their Cartier duals.

For example, the multiplication by $-1$ in $A$ lifts to a $G$-automorphism of $X$ if and only if 
there exists $z \in N_G(H)$ such that $\Int(z)$ acts on $H$ by the inverse $h \mapsto h^{-1}$.

\smallskip

\noindent
(ii) In characteristic $0$, any unipotent vector bundle over $A$ can be written as
$E(A) \times^H V$ for some representation $V$ of the vector group $H(A)$,
uniquely determined up to conjugacy. Moreover, given two unipotent vector bundles
$E_1$, $E_2$ and denoting by $V_1$, $V_2$ the associated representations, we have
$$
\Hom_A(E_1,E_2) = H^0(A,E_1^* \otimes E_2) \cong (\cO(E(A)) \otimes V_1^* \otimes V_2)^{H(A)}
\cong \Hom^{H(A)}(V_1,V_2).
$$
Thus, \emph{the assignement $V \mapsto E(A) \times^{H(A)} V$ yields an equivalence of categories 
from finite-dimensional representations of $H(A)$ to unipotent vector bundles}.

Since the category of finite-dimensional $H(A)$-modules is equivalent to that of
modules of finite length over the local ring $\cO_{\widehat{A},0}$, we recover Mukai's 
result that \emph{the category of unipotent vector bundles is equivalent to that 
of coherent sheaves over $\widehat{A}$ supported at $0$} (see \cite[Theorem 4.12]{Muk79}). 
Note however that Mukai's proof, via the Fourier transform, is valid in all characteristics.

\smallskip

\noindent
(iii) Consider homogeneous $G$-bundles over an elliptic curve $E$ in characteristic $0$. 
Since $H^1(E,\cO_E) \cong k$, the data of $H_u$ and $\delta$ as above are equivalent 
to that of a unipotent element of $G$ which generates $H_u$. Also, $E \cong \widehat{E}$. 
It follows that 
\emph{the homogeneous $G$-bundles over $E$ are classified by the triples $(u,H_s,c)$ 
where $u \in G$ is unipotent, $H_s$ is a diagonalizable subgroup of the centralizer $C_G(u)$, 
and $c : \widehat{H_s} \to E(k)$ is an injective homomorphism, up to simultaneous conjugacy 
of $u$ and $H_s$ in $G$}.

This classification still holds for homogeneous $G$-bundles over an \emph{ordinary} 
elliptic curve $E$ in characteristic $p > 0$. 
Indeed, $\widehat{H_u}$ is then a local subgroup scheme of $E$ of order some
power $p^r$, and hence is isomorphic to $\mu_{p^r}$; equivalently, $H_u \cong \bZ/p^r \bZ$.
So the data of $H_u$ and $\delta$ are again equivalent to that of a unipotent element of $G$. 

As a consequence, in either case, the unipotent vector bundles of rank $n$ over $E$ are 
classified by the unipotent conjugacy classes in $\GL_n$. In particular, there is a unique
indecomposable such bundle up to isomorphism: the Atiyah bundle, associated with the 
unipotent class having one Jordan block. But this indecomposable bundle, viewed as a 
homogeneous vector bundle, has very different structures in characteristic $0$ and in 
positive characteristics. We refer to \cite{Sc09} for further results on this topic,
and to \cite{FMW98, La98, FM00} for developments on (possibly non-homogeneous) holomorphic 
principal bundles over complex elliptic curves. 
\end{remarks}

\address{Universit\'e Grenoble I\\
Institut Fourier\\ 
UMR 5582, BP 74\\
38402 Saint-Martin d'H\`eres Cedex\\ 
France}
{Michel.Brion@ujf-grenoble.fr}

\end{document}